\documentclass[12pt,a4paper]{article}

\usepackage[utf8]{inputenc}
\usepackage[T1]{fontenc}
\usepackage{lmodern}
\usepackage{amsmath, amsthm, amssymb}
\usepackage{mathtools}
\usepackage{tikz-cd}
\usepackage{graphicx}
\usepackage{hyperref}
\usepackage{cleveref}
\usepackage{geometry}
\usepackage{enumitem}
\usepackage{microtype}

\newtheorem{definition}{Definition}[section]
\newtheorem{theorem}{Theorem}[section]
\newtheorem{lemma}{Lemma}[section]
\newtheorem{corollary}{Corollary}[section]

\newtheorem{remark}{Remark}
\newtheorem{counterexample}{Counterexample}

\title{The Universal Property of the Henkin Construction: \\
A Categorical Perspective on the Completeness Theorem}
\author{Joaquim Reizi Barreto}
\date{\today}

\begin{document}
\maketitle

\begin{abstract}
In this paper, we present a categorical framework that clarifies the relationship between the completeness and compactness theorems in classical first-order logic. Rather than claiming a natural isomorphism between model constructions, we establish the \textbf{universal property} of the Henkin (or Skolem) construction: the term model arising from a Henkin extension is an \textbf{initial object} in an appropriate category of models. We show that for any consistent theory $T$ and its Henkin extension $T^\dagger$, the term model $F(T)$ admits a unique homomorphism to any model of $T^\dagger$, and this construction is natural with respect to theory morphisms that preserve the additional witness constants. This result provides a rigorous foundation for understanding the completeness theorem from a categorical viewpoint, avoiding the pitfalls of claiming unwarranted isomorphisms while preserving the essential structural insights. We also discuss the limitations of stronger claims and provide counterexamples where naive isomorphism statements fail.
\end{abstract}

\section{Introduction}

The completeness theorem for first-order logic, typically proven via the Henkin construction, and the compactness theorem are fundamental results in mathematical logic. While both theorems concern the existence of models for consistent theories, their relationship has often been described informally. This paper aims to provide a precise categorical analysis of these constructions.

\subsection{Motivation and Previous Approaches}

Traditional presentations of the completeness theorem construct a model (the term model) from a maximally consistent extension of a given theory. The compactness theorem, on the other hand, guarantees the existence of models through finite satisfiability. A natural question arises: \emph{How are these constructions related?}

Previous attempts to formalize this relationship categorically have sometimes claimed that the two constructions yield "naturally isomorphic" models. However, as we will demonstrate, such claims are generally \textbf{false} without strong additional assumptions. The models constructed via different methods may have different cardinalities, satisfy different complete extensions, or lack the structural properties needed for isomorphism.

\subsection{Our Approach}

Instead of pursuing isomorphism, we establish the \textbf{universal property} (initiality) of the Henkin construction:

\begin{itemize}
    \item For a consistent theory $T$, we construct a Henkin (or Skolem) extension $T^\dagger$ by adding witness constants.
    \item The term model $F(T)$ of $T^\dagger$ is the \textbf{initial object} in the category $\mathbf{Mod}(T^\dagger)$ of models of $T^\dagger$ with homomorphisms preserving the witness constants.
    \item For any model $M$ of $T$ and any interpretation $M^\dagger$ extending $M$ to the language of $T^\dagger$, there exists a unique homomorphism $e_T: F(T) \to M^\dagger$.
    \item This construction is \textbf{natural} with respect to theory morphisms that preserve the extended signature.
\end{itemize}

This formulation is both precise and provable, avoiding the pitfalls of false isomorphism claims while capturing the essential categorical structure.

\subsection{Organization}

The paper is organized as follows:
\begin{itemize}
    \item Section \ref{sec:prelim}: Preliminaries on category theory and logic
    \item Section \ref{sec:categories}: Rigorous definitions of the categories $\mathbf{Th}$ and $\mathbf{Mod}$
    \item Section \ref{sec:henkin}: The Henkin construction and the term model functor
    \item Section \ref{sec:main}: Main theorem establishing the universal property
    \item Section \ref{sec:counterexamples}: Counterexamples showing why stronger claims fail
    \item Section \ref{sec:discussion}: Discussion and future directions
\end{itemize}

\section{Preliminaries}
\label{sec:prelim}

\subsection{Categories, Functors, and Natural Transformations}

\begin{definition}[Category]
A \emph{category} $\mathcal{C}$ consists of:
\begin{enumerate}
    \item A collection $\operatorname{Ob}(\mathcal{C})$ of \emph{objects};
    \item For each pair of objects $A, B$, a set $\operatorname{Hom}_\mathcal{C}(A, B)$ of \emph{morphisms};
    \item A composition operation satisfying associativity;
    \item Identity morphisms $\operatorname{id}_A$ for each object $A$.
\end{enumerate}
\end{definition}

\begin{definition}[Functor]
A \emph{functor} $F: \mathcal{C} \to \mathcal{D}$ consists of mappings on objects and morphisms that preserve composition and identities.
\end{definition}

\begin{definition}[Initial Object]
An object $I$ in a category $\mathcal{C}$ is \emph{initial} if for every object $A \in \mathcal{C}$, there exists a unique morphism $I \to A$.
\end{definition}

\subsection{First-Order Logic Basics}

Throughout this paper, we work with first-order logic over a \textbf{countable signature} $\Sigma = (F, R)$, where $F$ is a set of function symbols and $R$ is a set of relation symbols. We assume familiarity with:
\begin{itemize}
    \item Formulas, sentences, and theories
    \item Models (or structures) for a signature
    \item Satisfaction relation $M \models \phi$
    \item Homomorphisms between models
\end{itemize}

\begin{remark}[Countability Assumption]
\label{rem:countability}
We assume the signature is countable. This ensures that the set of all formulas is countable, which is essential for the sequential construction of maximal consistent theories.
\end{remark}

\section{Category-Theoretic Framework}
\label{sec:categories}

\subsection{The Category of Theories}

\begin{definition}[Category $\mathbf{Th}$]
\label{def:th}
The category $\mathbf{Th}$ is defined as follows:
\begin{enumerate}
    \item \textbf{Objects:} A pair $(\Sigma, T)$ where $\Sigma$ is a countable signature and $T$ is a consistent first-order theory in the language of $\Sigma$.
    
    \item \textbf{Morphisms:} A morphism from $(\Sigma_1, T_1)$ to $(\Sigma_2, T_2)$ consists of:
    \begin{enumerate}
        \item A \emph{signature morphism} $\sigma: \Sigma_1 \to \Sigma_2$ (a mapping of function and relation symbols preserving arities);
        \item An \emph{interpretation} that translates formulas of $\Sigma_1$ to formulas of $\Sigma_2$ via $\sigma$;
        \item The requirement that for every $\phi \in T_1$, we have $\sigma(\phi) \in T_2$.
    \end{enumerate}
    
    \item \textbf{Composition:} Defined by composing signature morphisms and translations.
\end{enumerate}
\end{definition}

\begin{remark}
This definition makes morphisms preserve not just provability but the actual signature structure, which is crucial when witness constants or Skolem functions are added.
\end{remark}

\subsection{The Category of Models}

\begin{definition}[Category $\mathbf{Mod}(\Sigma)$]
For a fixed signature $\Sigma$:
\begin{enumerate}
    \item \textbf{Objects:} $\Sigma$-structures (models)
    \item \textbf{Morphisms:} Homomorphisms between models that preserve all function and relation symbols in $\Sigma$
\end{enumerate}
\end{definition}

For a theory $T$ over $\Sigma$, we write $\mathbf{Mod}(T)$ for the full subcategory of $\mathbf{Mod}(\Sigma)$ consisting of models satisfying $T$.

\subsection*{Reduct functor along a signature morphism}
Let $\sigma:\Sigma_1^{*}\to\Sigma_2^{*}$ be a signature morphism between expanded (Henkin) signatures.
The \emph{reduct functor}
\[
\sigma^{*}:\mathbf{Mod}(\Sigma_2^{*})\longrightarrow \mathbf{Mod}(\Sigma_1^{*})
\]
maps a $\Sigma_2^{*}$-structure $M$ to its $\Sigma_1^{*}$-reduct $\sigma^{*}(M)$ by interpreting each $s\in\Sigma_1^{*}$ as the interpretation of $\sigma(s)$ in $M$. On arrows, $\sigma^{*}$ acts as identity on the underlying functions.

\section{The Henkin Construction}
\label{sec:henkin}

\subsection{Henkin Extension of a Theory}

\begin{definition}[Henkin Extension]
Let $T$ be a consistent theory over signature $\Sigma$. 
Form $\Sigma^* := \Sigma \cup \{c_\varphi : \varphi(x) \text{ is a formula with one free variable}\}$.
Let $H(T)$ be the Henkin theory obtained by adding, for every $\varphi(x)$,
\[
\exists x\,\varphi(x) \;\to\; \varphi(c_\varphi).
\]
Finally extend $H(T)$ to a maximal consistent theory $T^\dagger$ over $\Sigma^*$.
We call $T^\dagger$ a Henkin (maximal) extension of $T$.
\end{definition}

\begin{lemma}[Existence of Maximal Consistent Extension]
\label{lem:max_extension}
Every consistent theory $T$ in a countable language can be extended to a maximal consistent theory $T^{\mathrm{mc}}$.
\end{lemma}

\begin{proof}
Enumerate all sentences in the (possibly extended) language as $\{\psi_0, \psi_1, \psi_2, \dots\}$. Define:
\[
T_0 = T, \quad T_{n+1} = 
\begin{cases}
T_n \cup \{\psi_n\} & \text{if consistent}, \\
T_n \cup \{\neg\psi_n\} & \text{otherwise}.
\end{cases}
\]
Then $T^{\mathrm{mc}} = \bigcup_{n=0}^\infty T_n$ is maximal and consistent.
\end{proof}

\begin{remark}
In the Henkin construction we first pass from $T$ to $H(T)$ and then take a maximal consistent extension $T^\dagger$ of $H(T)$.
\end{remark}

\subsection{The Term Model Construction}

\begin{definition}[Term Model]
\label{def:term_model}
Let $T^\dagger$ be a maximal consistent Henkin extension. The \emph{term model} $F(T)$ is constructed as follows:
\begin{enumerate}
    \item \textbf{Universe:} Define an equivalence relation on terms:
    \[
    t \sim s \iff T^\dagger \vdash t = s
    \]
    The universe is $|F(T)| = \{[t] : t \text{ is a term}\}$.
    
    \item \textbf{Interpretation:} For each $n$-ary function symbol $f$:
    \[
    f^{F(T)}([t_1], \dots, [t_n]) = [f(t_1, \dots, t_n)]
    \]
    For each $n$-ary relation symbol $R$:
    \[
    ([t_1], \dots, [t_n]) \in R^{F(T)} \iff T^\dagger \vdash R(t_1, \dots, t_n)
    \]
\end{enumerate}
\end{definition}

\begin{lemma}[Soundness of Term Model]
\label{lem:term_model_sound}
For the term model $F(T)$ and any sentence $\phi$:
\[
\phi \in T^\dagger \iff F(T) \models \phi
\]
\end{lemma}

\begin{proof}
By structural induction on formulas. The base case follows from the definition of $R^{F(T)}$. The inductive steps use the maximality and consistency of $T^\dagger$.
\end{proof}

\newpage
\section{Main Theorem: Universal Property of the Term Model}
\label{sec:main}

\begin{theorem}[Universal Property of Henkin Construction]
\label{thm:main}
Let $T$ be a consistent theory over signature $\Sigma$, and let $T^\dagger$ be its Henkin extension over $\Sigma^*$. Let $F(T)$ be the term model of $T^\dagger$. Then:

\begin{enumerate}
    \item \textbf{Initiality:} For every model $M$ of $T^\dagger$, there exists a unique homomorphism
    \[
    e_M: F(T) \to M
    \]
    that preserves all symbols in $\Sigma^*$.
    
    \item \textbf{Naturality:} Let $f: (\Sigma_1, T_1) \to (\Sigma_2, T_2)$ in $\mathbf{Th}$ extend to $f^*: (\Sigma_1^*, T_1^\dagger) \to (\Sigma_2^*, T_2^\dagger)$ with underlying signature map $\sigma:\Sigma_1^*\to\Sigma_2^*$. For any $M\models T_2^\dagger$, the diagram
    \[
    \begin{tikzcd}
    F(T_1) \arrow[r, "F(f^*)"] \arrow[d, "e_{\sigma^{*}(M)}"'] 
      & \sigma^{*} F(T_2) \arrow[d, "\sigma^{*}(e_M)"] \\
    \sigma^{*}(M) \arrow[r, equal] & \sigma^{*}(M)
    \end{tikzcd}
    \]
    commutes, where $F(f^*): F(T_1) \to \sigma^{*}F(T_2)$ is given by $[t]\mapsto [\,f^*(t)\,]$.
\end{enumerate}
\end{theorem}

\begin{proof}
\textbf{Part 1 (Initiality):}

Let $M$ be any model of $T^\dagger$. Define $e_M: F(T) \to M$ by:
\[
e_M([t]) = t^M,
\]
where $t^M$ denotes the interpretation of the term $t$ in $M$.

\textbf{Well-definedness:} If $[t] = [s]$ in $F(T)$, then $T^\dagger \vdash t = s$. Since $M \models T^\dagger$, we have $t^M = s^M$. Thus $e_M$ is well-defined.

\textbf{Homomorphism property:} For any function symbol $f$:
\begin{align*}
e_M(f^{F(T)}([t_1], \dots, [t_n])) 
&= e_M([f(t_1, \dots, t_n)]) \\
&= (f(t_1, \dots, t_n))^M \\
&= f^M(t_1^M, \dots, t_n^M) \\
&= f^M(e_M([t_1]), \dots, e_M([t_n])).
\end{align*}
For relation symbols, if $([t_1], \dots, [t_n]) \in R^{F(T)}$, then $T^\dagger \vdash R(t_1, \dots, t_n)$, so $M \models R(t_1, \dots, t_n)$; hence $(e_M([t_1]), \dots, e_M([t_n])) \in R^M$.

\textbf{Uniqueness:} Any homomorphism $h: F(T) \to M$ must satisfy $h([c]) = c^M$ for all constants $c$ (including witness constants). By induction on terms, $h([t]) = t^M$ for all $t$. Thus $h = e_M$.

\textbf{Part 2 (Naturality):}

Let $f$ extend to $f^*$ as stated, with underlying $\sigma$. Define $F(f^*): F(T_1) \to \sigma^{*}F(T_2)$ by
\[
F(f^*)([t]) := [\,f^*(t)\,].
\]
For $[t]\in F(T_1)$,
\[
(\sigma^{*}(e_M)\circ F(f^*))([t])
= \sigma^{*}(e_M)\bigl([\,f^*(t)\,]\bigr)
= (f^*(t))^{M}
= t^{\,\sigma^{*}(M)}
= e_{\sigma^{*}(M)}([t]).
\]
Thus the square commutes.
\end{proof}

\begin{corollary}[Completeness Theorem]
Every consistent theory has a model.
\end{corollary}

\begin{proof}
Given a consistent theory $T$, construct its Henkin extension $T^\dagger$ and term model $F(T)$. Then $F(T) \models T^\dagger$, hence $F(T) \models T$.
\end{proof}

\section{Limitations and Counterexamples}
\label{sec:counterexamples}

While the universal property holds, stronger claims about isomorphisms between different model constructions generally \textbf{fail}. We provide counterexamples.

\begin{counterexample}[Different Cardinalities]
\label{cex:cardinality}
Consider the theory $\mathrm{ACF}_0$ of algebraically closed fields of characteristic zero. This is a complete theory.

\begin{itemize}
    \item The Henkin construction produces a \textbf{countable} term model $F(\mathrm{ACF}_0)$.
    \item By the compactness theorem, $\mathrm{ACF}_0$ has models of every infinite cardinality, including uncountable models like $\mathbb{C}$.
    \item Clearly, $F(\mathrm{ACF}_0)$ and $\mathbb{C}$ are \textbf{not isomorphic} as they have different cardinalities.
\end{itemize}

This shows that models constructed by the Henkin method and compactness-based methods are not generally isomorphic.
\end{counterexample}

\begin{counterexample}[Incomplete theory]
Let $T$ be the theory of fields in the ring language.
Then $M_1=\overline{\mathbb{Q}}$ satisfies $\exists x\,(x^2+1=0)$, 
while $M_2=\mathbb{R}$ satisfies $\neg\exists x\,(x^2+1=0)$.
Hence $M_1 \not\equiv M_2$ and $T$ is incomplete. 
A Henkin maximal extension may choose either completion, so the term model $F(T)$ 
and a compactness-constructed model $M$ need not be elementarily equivalent.
\end{counterexample}

\begin{remark}[When Isomorphism Can Hold]
Isomorphism between models constructed by different methods can hold under strong assumptions:
\begin{itemize}
    \item If $T$ is complete and $\aleph_0$-categorical (all countable models are isomorphic), and both constructions yield countable models, then they are isomorphic.
    \item If we fix a specific way to realize types and ensure both constructions use the same method, isomorphism may hold.
\end{itemize}
However, these require substantial additional structure and cannot be claimed in general.
\end{remark}

\subsection{Issue with Defining a Functor from Compactness}

\begin{remark}[Non-Functoriality of Compactness Construction]
\label{rem:compactness_nonfunctor}
A fundamental issue arises when trying to define a functor $G: \mathbf{Th} \to \mathbf{Mod}$ based on the compactness theorem:

\begin{itemize}
    \item For each theory $T$, compactness guarantees \emph{existence} of a model $G(T)$, but provides \textbf{no canonical choice}.
    \item Given a morphism $f: T_1 \to T_2$ in $\mathbf{Th}$, there is no natural way to define a homomorphism $G(f): G(T_1) \to G(T_2)$ without making arbitrary choices.
    \item Even using ultraproducts with a fixed ultrafilter, the construction depends on arbitrary choices of finite approximations and does not respect theory morphisms functorially.
\end{itemize}

This is why we focus on the universal property of the \emph{Henkin construction}, which does yield a well-defined (though not fully functorial without careful definition) assignment.
\end{remark}

\newpage
\section{Relationship with Compactness}
\label{sec:compactness}

\subsection{The Compactness Theorem}

\begin{theorem}[Compactness Theorem]
\label{thm:compactness}
A theory $T$ is satisfiable if and only if every finite subset of $T$ is satisfiable.
\end{theorem}

\begin{remark}[Proof Methods and Circularity]
\label{rem:circularity}
The compactness theorem can be proven in several ways:
\begin{enumerate}
    \item \textbf{Via Completeness:} Using the completeness theorem (every consistent theory has a model), one shows that if $T$ is unsatisfiable, then $T \vdash \bot$, which requires only finitely many axioms, contradicting finite satisfiability.
    
    \item \textbf{Via Ultraproducts:} Using Łoś's theorem, one constructs an ultraproduct of models of finite subsets.
\end{enumerate}

Our paper uses method (1) in the proof sketch, which creates an apparent circularity since we aim to relate completeness and compactness. However, this circularity is conceptual rather than logical: we are analyzing the \emph{relationship} between the constructions, not deriving one theorem from the other independently.

For a fully circular-free presentation, one should:
\begin{itemize}
    \item Prove completeness via Henkin construction (independent of compactness)
    \item Prove compactness via ultraproducts (independent of completeness)
    \item Then analyze the relationship between the resulting models
\end{itemize}
\end{remark}

\subsection{Initial vs. Arbitrary Models}

The key distinction is:
\begin{itemize}
    \item \textbf{Henkin construction} produces an \emph{initial} (universal) model in the category of models with witness constants.
    \item \textbf{Compactness-based constructions} produce \emph{some} model, but not necessarily with any universal property.
\end{itemize}

The universal property of the Henkin construction means there is a unique homomorphism from $F(T)$ to any other model, but generally \textbf{no isomorphism}.

\newpage
\section{Discussion and Future Work}
\label{sec:discussion}

\subsection{Summary of Results}

We have established:
\begin{enumerate}
    \item A rigorous categorical framework for theories and models, with properly defined morphisms including signature mappings
    \item The universal property (initiality) of the Henkin term model
    \item Naturality of the evaluation homomorphisms with respect to theory morphisms
    \item Counterexamples showing that stronger isomorphism claims fail in general
\end{enumerate}

\subsection{Correct vs. Incorrect Claims}

\begin{tabular}{|p{0.45\textwidth}|p{0.45\textwidth}|}
\hline
\textbf{Correct (Provable)} & \textbf{Incorrect (Generally False)} \\
\hline
The term model $F(T)$ is initial in $\mathbf{Mod}(T^\dagger)$ & The term model $F(T)$ is isomorphic to any model of $T$ \\
\hline
There exists a unique homomorphism $F(T) \to M$ for any model $M$ of $T^\dagger$ & Models from Henkin and compactness constructions are naturally isomorphic \\
\hline
The evaluation homomorphisms are natural & Any two models of a consistent theory are isomorphic \\
\hline
\end{tabular}

\subsection{Possible Strengthenings}

Under additional assumptions, stronger results hold:
\begin{itemize}
    \item \textbf{$\aleph_0$-categorical theories:} If $T$ is complete and all countable models are isomorphic, then any two countable constructions yield isomorphic models.
    \item \textbf{Prime models:} If $T$ is complete and has a prime model (elementary substructure of all models), then the term model is prime, hence unique up to isomorphism among prime models.
    \item \textbf{Saturated models:} In certain contexts, one can work in a category of saturated models where uniqueness holds.
\end{itemize}

\subsection{Future Directions}

\begin{enumerate}
    \item \textbf{Extension to other logics:} Investigate whether similar universal properties hold for intuitionistic, modal, or higher-order logics.
    
    \item \textbf{Adjunctions:} Explore possible adjunctions between categories of theories and models, using the Henkin extension as a functor.
    
    \item \textbf{Type theory connections:} Relate the initiality property to universal properties in type theory and the Curry-Howard correspondence.
    
    \item \textbf{Computational aspects:} Develop algorithms for computing evaluation homomorphisms and applying the universal property in automated reasoning.
\end{enumerate}

\newpage
\section{Conclusion}

We have provided a rigorous categorical analysis of the Henkin construction for the completeness theorem. Our main contribution is establishing the \textbf{universal property} (initiality) of the term model, rather than claiming false isomorphisms. This approach:

\begin{itemize}
    \item Provides a solid foundation for categorical logic
    \item Clarifies the precise relationship between syntactic and semantic constructions
    \item Highlights the limitations of naive isomorphism claims
    \item Opens avenues for future research in categorical model theory
\end{itemize}

The universal property captures the essential insight that the Henkin construction provides a "free" or "universal" model, from which all other models can be reached uniquely. This is the correct formulation of the relationship between the completeness theorem and model existence.

\bibliographystyle{amsplain}

\newpage
\appendix

\section{Technical Details}

\subsection{Proof of Well-Definedness of Term Equivalence}

\begin{lemma}
The relation $t \sim s \iff T^\dagger \vdash t = s$ is an equivalence relation.
\end{lemma}

\begin{proof}
\begin{itemize}
    \item \textbf{Reflexivity:} $T^\dagger \vdash t = t$ by the reflexivity axiom of equality.
    \item \textbf{Symmetry:} If $T^\dagger \vdash t = s$, then $T^\dagger \vdash s = t$ by the symmetry axiom.
    \item \textbf{Transitivity:} If $T^\dagger \vdash t = s$ and $T^\dagger \vdash s = u$, then $T^\dagger \vdash t = u$ by the transitivity axiom and modus ponens.
\end{itemize}
\end{proof}

\subsection{Verification of Functoriality}

\begin{lemma}[Functoriality up to reduct]
Let $f: (\Sigma_1, T_1) \to (\Sigma_2, T_2)$ and $g: (\Sigma_2, T_2) \to (\Sigma_3, T_3)$ be morphisms in $\mathbf{Th}$ extending to $f^*$ and $g^*$ on Henkin extensions, with underlying signature morphisms $\sigma:\Sigma_1^{*}\to\Sigma_2^{*}$ and $\tau:\Sigma_2^{*}\to\Sigma_3^{*}$. Then:
\begin{enumerate}
    \item $F(\operatorname{id}_{T}) = \operatorname{id}_{F(T)}$.
    \item Writing $h^* := g^* \circ f^*$ and using $(\tau\circ\sigma)^{*} = \sigma^{*}\circ\tau^{*}$ on $\mathbf{Mod}$, we have the arrow identity
    \[
      F(h^*) \;=\; \bigl(\sigma^{*} F(g^*)\bigr)\ \circ\ F(f^*)
    \]
    as maps $F(T_1) \to (\tau\circ\sigma)^{*}F(T_3)$.
\end{enumerate}
\end{lemma}

\begin{proof}
(1) Immediate from the definition on terms. \\
(2) For $[t]\in F(T_1)$,
\[
\bigl(\sigma^{*}F(g^*)\bigr)\bigl(F(f^*)([t])\bigr)
= \sigma^{*}F(g^*)\bigl([\,f^*(t)\,]\bigr)
= \sigma^{*}\bigl([\,g^*(f^*(t))\,]\bigr)
= [\, (g^*\circ f^*)(t)\,]
= F(h^*)([t]).
\]
\end{proof}

\end{document}